\documentclass{amsart}
\usepackage{amsmath}
\usepackage{paralist}
\usepackage{url}
\usepackage{mathrsfs}
\usepackage{mathtools}
\usepackage{amsthm}
\usepackage{amssymb}
\usepackage{hyperref}
\usepackage{enumerate}

\newtheorem{theorem}{Theorem}[section]

\newtheorem{defi}[theorem]{Definition}

\newtheorem{example}[theorem]{Example}

\newtheorem{remark}[theorem]{Remark}

\numberwithin{equation}{section}

\newcommand{\re}{\mathbb{R}}

\newcommand{\F}{\mathbb{F}}
\newcommand{\N}{\mathbb{N}}

\newcommand{\half}{\frac{1}{2}}

\newcommand{\nn}{\nonumber}

\def\af{\alpha}

\def\rank{\mbox{rank}}

\newcommand{\reff}[1]{(\ref{#1})}
\renewcommand{\vec}[1]{\mbox{vec}(#1)}

\newcommand{\mc}[1]{\mathcal{#1}}

\newcommand{\supp}[1]{\mbox{supp}(#1)}
\newcommand{\st}{\mathrm{s.t.}}

\newcommand{\ideal}[1]{\mathit{Ideal}[#1]}
\newcommand{\qmod}[1]{\mathit{QM}[#1]}
\newcommand{\mt}[1]{\mathtt{#1}}
\newcommand{\mA}{\mathcal{A}}

\newcommand{\mB}{\mathcal{B}}

\newcommand{\bdes}{\begin{description}}
\newcommand{\edes}{\end{description}}

\newcommand{\bal}{\begin{align}}
\newcommand{\eal}{\end{align}}

\newcommand{\bnum}{\begin{enumerate}}
\newcommand{\enum}{\end{enumerate}}

\newcommand{\bit}{\begin{itemize}}
\newcommand{\eit}{\end{itemize}}

\newcommand{\bea}{\begin{eqnarray}}
\newcommand{\eea}{\end{eqnarray}}
\newcommand{\be}{\begin{equation}}
\newcommand{\ee}{\end{equation}}

\newcommand{\baray}{\begin{array}}
\newcommand{\earay}{\end{array}}

\newcommand{\bsry}{\begin{subarray}}
\newcommand{\esry}{\end{subarray}}

\newcommand{\bca}{\begin{cases}}
\newcommand{\eca}{\end{cases}}

\newcommand{\bcen}{\begin{center}}
\newcommand{\ecen}{\end{center}}

\newcommand{\bbm}{\begin{bmatrix}}
\newcommand{\ebm}{\end{bmatrix}}

\newcommand{\bmx}{\begin{matrix}}
\newcommand{\emx}{\end{matrix}}

\newcommand{\bpm}{\begin{pmatrix}}
\newcommand{\epm}{\end{pmatrix}}

\newcommand{\bvm}{\begin{vmatrix}}
\newcommand{\evm}{\end{vmatrix}}

\newcommand{\btab}{\begin{tabular}}
\newcommand{\etab}{\end{tabular}}

\begin{document}

\title{The Rank-$1$ Completion Problem for Cubic Tensors}

\author[Jinling~Zhou]{Jinling Zhou}
\author[Jiawang~Nie]{Jiawang Nie}
\author[Zheng~Peng]{Zheng Peng}
\author[Guangming~Zhou]{Guangming Zhou}

\address{
Jinling~Zhou, Zheng Peng, Guangming Zhou,
School of Mathematics and Computational Science,
Xiangtan University, Xiangtan, Hunan, 411105, China.}
\email{jinlingzhou@smail.xtu.edu.cn,
pzheng@xtu.edu.cn,
zhougm@xtu.edu.cn}

\address{Jiawang Nie, Department of Mathematics,
University of California San Diego,
9500 Gilman Drive, La Jolla, CA, USA, 92093.}
\email{njw@math.ucsd.edu}

\date{}

\begin{abstract}
This paper studies the rank-$1$ tensor completion problem
for cubic tensors.
First of all, we show that this problem is equivalent
to a special rank-$1$ matrix recovery problem.
When the tensor is strongly rank-$1$ completable, we show that
the problem is equivalent to a rank-$1$ matrix completion problem
and it can be solved by an iterative formula. For other cases,
we propose both nuclear norm relaxation and moment relaxation methods
for solving the resulting rank-$1$ matrix recovery problem.
The nuclear norm relaxation sometimes returns a rank-$1$ tensor completion,
while sometimes it does not. When it fails,
we apply the moment hierarchy of semidefinite programming relaxations
to solve the rank-$1$ matrix recovery problem.
The moment hierarchy can always get a rank-$1$ tensor completion,
or detect its nonexistence.
Numerical experiments are shown to demonstrate
the efficiency of these proposed methods.
\end{abstract}





\maketitle

\section{Introduction}
\label{sec:introduction}

Let $\mathbb{F}$ be the real filed $\mathbb{R}$ or complex filed $\mathbb{C}$.
A tensor $\mathcal{A} \in \F^{n_1\times\cdots\times n_m}$
can be viewed as the multi-array indexed such that
\[
\mathcal{A}=(\mathcal{A}_{i_1\cdots i_m})_{
\substack{
 1 \le i_1 \le n_1, \ldots,
 1 \le i_m \le n_m  .
}  }
\]
The integer $m$ is called the $\mathnormal{order}$ of $\mathcal{A}$.
When $m=3$, we call $\mA$ a cubic tensor.
For vectors $u_k\in \F^{n_k}$, $k=1,\cdots ,m$,
the notation $u_1\otimes\cdots\otimes u_m$
denotes the tensor in $\F^{n_1\times\cdots\times n_m}$ such that
\[
( u_1\otimes \cdots \otimes u_m   )_{i_1\cdots i_m} =
( u_1 )_{i_1} \cdots  ( u_m )_{i_m} ,
\]
for all indices $i_1,\cdots,i_m$. Tensors of the form
$
u_1\otimes \cdots \otimes u_m
$
are called rank-$1$ tensors. When $n_1=\cdots =n_m = n$,
tensors in $\mathbb{F}^{n_1\times\cdots\times n_m}$
are said to be {\it hypercubical}. We denote the hypercubical tensor space
\[
\mathtt{T}^m(\F^n) \,  \coloneqq \,  \F^n\otimes \cdots \otimes \F^n
\qquad \text{($\F^n$ appears $m$ times)}.
\]
When it is hypercubical, the tensor $\mathcal{A}$
is said to be symmetric if
$
\mathcal{A}_{i_1\cdots i_m}=\mathcal{A}_{j_1\cdots j_m}
$
for every permutation $(i_1\cdots i_m)$ of $(j_1\cdots j_m)$.
The subspace of symmetric tensors in $\mathtt{T}^m(\F^n)$
is denoted as $\mathtt{S}^m(\F^n)$.
The dimension of $\F^{n_1\times\cdots\times n_m}$ is $n_1 \cdots  n_m$,
while the dimension of $\mathtt{S}^m(\F^n)$ is $\binom{n+m-1}{m}$.

An important concept for tensors is rank.
For $\mathcal{A} \in \F^{n_1\times\cdots\times n_m}$,
its rank is the smallest $r$ such that
\[
\mc{A} = \mA_1 + \cdots + \mA_r,
\]
where each $\mA_i \in \F^{n_1\times\cdots\times n_m}$ is of rank $1$.
The above equation is also often referenced as
Candecomp-Parafac (CP) decomposition
and the smallest $r$ is referred as CP rank
in the literature \cite{hitchcock1927}.
We refer to \cite{Breiding,carroll1970,LMV2004,Harshman,GPSTD,NWZ23,SBL13,TelenVan22}
for the work of computing tensor decompositions.
It is interesting to remark that the tensor rank
depends on the ground field \cite{Lim13}.
There also exist other types of tensor ranks, such as border rank
and multi-linear rank.
We refer to the work \cite{LMV2000,hitchcock1928,Land12,Lim13}
for various types of ranks for tensors.

The tensor completion problem (TCP) is to find values for missing entries
of a partially given tensor so that it has certain property, e.g., it has low rank.
TCP has broad applications, such as
computer vision \cite{L-W-13,qin2022,qiu2021},
recommendation systems \cite{frolov2017,Kara-10},
imaging and signal processing \cite{L-H-14,zhao2020}.
We refer to \cite{KolBad09} for more applications of tensors.

Tensor completion is an extension of matrix completion.
The nuclear norm relaxation is frequently used to get matrix completions
\cite{cai2010,candes2012,recht2010}.
Nuclear norms are also defined for tensors \cite{FriLim18,Nie-nuclear17}.
Tensor nuclear norm relaxations can be extended
to solve tensor completion or recovery problems
\cite{MHWG,tang2015,tian2024tensor,YuaZha16}.
Optimization methods based on CP decompositions can be used to
get low rank tensor completions or recovery
\cite{Bai2016,qiu2021,zhao2020}.
Tucker decomposition based techniques are also applicable to get
tensor completions via minimizing ranks of unfolded matrices
\cite{L-W-13,rauhut2017}.
Recently, there are Riemannian-manifold optimization methods
for getting tensor completions
\cite{Dong-Gao-21,gao2024riemannian,K-S-V-14,Swijsen22}.
%
%

Many tensor optimization problems can be formulated as
moment or polynomial optimization (see \cite[Chapter~11]{mpo}).
Tensor nuclear norms can be computed by solving
moment relaxations \cite{Nie-nuclear17,tang2015}.
%
%
Real eigenvalues of tensors can be obtained by solving
certain polynomial optimization problems \cite{CDN14,NieZhang18}.
Stable rank-one matrix completions can be solved by
the level-$2$ Moment-SOS relaxation \cite{C-D-17}.

\subsection*{Contributions}

This paper focuses on rank-$1$ completions for cubic tensors.
It looks for a rank-$1$ tensor that has given values for some of its entries.
Consider a partially given tensor $\mA \in  \re^{n_1\times n_2\times n_3}$,
i.e., there is a subset
\[
\Omega  \, \subseteq  \,  [ n_1 ] \times [ n_2 ] \times [ n_3 ]
\]
such that the entry $\mA_{ijk}$ is given for all $(i,j,k) \in \Omega$.
The rank-$1$ tensor completion problem is to look for a rank-$1$ tensor
$a \otimes b \otimes c$ such that
$a_i b_j c_k = \mA_{ijk}$ for all $(i,j,k) \in \Omega$.
So, the problem can formulated as
\begin{equation}\label{eq-general-tc-r1}
\left\{ \baray{cl}
\text{find}  &  (a, b, c) \in  \re^{n_1} \times \re^{n_2} \times \re^{n_3}  \\
\st  &   \mA_{ijk} = a_i b_j c_k,   \, \,  (i,j,k) \in \Omega.
\earay \right.
\end{equation}
For the above $a,b,c$, the equation $\mA = a \otimes b \otimes c$
is called a rank-$1$ tensor completion for $\mA$.
The geometric properties of the rank-$1$ tensor completion problem
are well studied by Kahle et al. \cite{TKMZ}.
%
%
We remark that finding a rank-$1$ tensor completion
is equivalent to solving the polynomial optimization problem:
\[
\min_{a,b,c} \quad \sum_{(i,j,k)\in\Omega}(\mathcal{A}_{ijk}-a_ib_jc_k)^2.
\]
It has $3$ vector variables $a,b,c$ and the polynomial has degree six.
However, solving the above polynomial optimization directly is not computationally attractive.
This is shown in Example~\ref{nls:sos}.

In this paper, we show that the rank-$1$ tensor completion problem
is equivalent to a special rank-$1$ matrix recovery problem.
To be precise, we show that \eqref{eq-general-tc-r1}
is equivalent to
\begin{equation}  \label{eq-quadratic-sec1}
\left\{ \baray{cl}
\text{find}  & X = ab^T  \in  \re^{n_1 \times n_2 } \\
\st  & \mA_{i_1 j_1 k} X_{i_2 j_2} -\mA_{i_2 j_2 k} X_{i_1 j_1}=0  \\
 &  \qquad \text{for} \quad (i_1,\,j_1,\,k), (i_2, j_2, k) \in \Omega, \, k = 1, \ldots, n_3.
\earay \right.
\end{equation}
The equality constraint in \reff{eq-quadratic-sec1}
is equivalent to the determinantal equation
\[
\det \bbm
\mA_{i_1 j_1 k} & X_{i_1 j_1}\\
\mA_{i_2 j_2 k} &  X_{i_2 j_2}
\ebm=0 .
\]
The problem \reff{eq-quadratic-sec1}
is a special class of rank-$1$ matrix recovery.
The relationship between \reff{eq-general-tc-r1} and \reff{eq-quadratic-sec1}
is studied in Subsection~\ref{ssc:red2mat}.

A specially interesting case arises when
the values of $X_{ij}$ for $(i,j,k) \in \Omega$
are uniquely determined (up to scaling) by
equality constraints in \reff{eq-quadratic-sec1}.
For this case, the tensor completion problem can be reduced to
the rank-$1$ matrix completion problem
and the tensor $\mathcal{A}$ is said to be strongly rank-1 completable
(see Definition~\ref{de-strongly}),
if a rank-$1$ tensor completion exists.
Furthermore, if the corresponding bipartite graph is connected,
the rank-$1$ matrix completion problem can be solved
by an iterative formula.
These results are shown in Subsection~\ref{ssc:strongrk1}.

When $\mathcal{A}$ is not strongly rank-1 completable,
a natural approach for solving \reff{eq-quadratic-sec1}
is to apply the nuclear norm relaxation, shown in Subsection~\ref{ssc:nuclear}.
The nuclear norm relaxation may fail to find a rank-$1$ tensor completion.
However, we can always get a rank-$1$ tensor completion
by solving the moment hierarchy of semidefinite programming relaxations,
or we detect its nonexistence.
This is shown in Section~\ref{sec:tc-mom}.
When the tensor $\mA$ is symmetric, there are more attractive properties
for the nuclear norm and moment relaxations.
This is shown in Section~\ref{sc:sym}.
Our proposed method for cubic tensors can be naturally
extended to higher order tensors, shown in Section~\ref{sc:highord}.
Numerical experiments are provided in Section~\ref{sc:num}.

\section{Preliminaries}\label{sec:pre}
\subsection{Notation}

The symbol $\mathbb{N}$ (resp., $\mathbb{R}$, $\mathbb{C}$)
denotes the set of nonnegative integers (resp., real, complex numbers).
For integer $k>0$, $[k]$ stands for the set $\left\{ 1,\ldots, k\right\}$.
%
%
For a finite set $T$, its cardinality is denoted as $|T|$.
We denote by $\re[x]$ the ring of polynomials in
$x=(x_1,\cdots, x_n)$ and with real coefficients.
The subset of polynomials of degree $d$ in $\re[x]$
is denoted by $\re[x]_d$.
The superscript $^T$ denotes the transpose of a matrix or vector.
For a symmetric matrix $X$, the inequality $X \succeq 0$ (resp., $X \succ 0$)
means that $X$ is positive semidefinite (resp., positive definite).
The cone of all $N$-by-$N$ real symmetric positive semidefinite matrices
is denoted as $\mc{S}_+^N$.
A subset $I$ $\subseteq \mathbb{R} \left [ x \right ]$
is called an ideal if $pq \in I$ for all $p\in I$
and for all $q \in \mathbb{R}\left [ x \right ]$,
and $a+b\in I$ for all $a,b\in I$. For a tuple $h=(h_1,\cdots,h_m)$
of polynomials in $\mathbb{R}\left [ x \right ]$, we denote the ideal
\[
\ideal{h}  \coloneqq
h_1\cdot\mathbb{R}[x]+\cdots+h_m\cdot\mathbb{R}[x].
\]
For a set of polynomials $P \subseteq \mathbb{R}[x]$,
its real variety is the set
\[
V_{\mathbb{R}}(P)  \, \coloneqq  \,
\left\{x\in \mathbb{R}^n:\, p(x)=0\, \, \forall p\in P \right\} .
\]
A polynomial $p$ is said to vanish on a set $T$ if $p(u)=0$ for all $u\in T$.
If so, we write $p \equiv 0$ on $T$ or $p\vert_{T} \equiv 0$.

For $\alpha=(\alpha_1, \ldots, \alpha_n)\in\mathbb{N}^n$, define
$\left|\alpha  \right|  \coloneqq  \alpha_1+\cdots+\alpha_n$.
We denote the power set
\[
\mathbb{N}_d^n \coloneqq  \left\{ \alpha \in \N^n:\, \left| \alpha \right| \leq d\right\}.
\]
For $x=(x_1,\ldots,x_n)$, we denote the monomial power
$
x^{\alpha} \coloneqq  x_1^{\alpha_1}\cdots x_n^{\alpha_n}.
$
The column vector of all monomials in $x$ and of degrees up to $d$ is denoted
\[
\left [  x\right ]_{d}  \coloneqq
\left [  1\quad x_1\quad \cdots \quad x_n\quad x_1^2\quad x_1x_2\quad\cdots \quad x_n^d\right ]^{T}.
\]
The length of the vector $\left [  x\right ]_{d}$ is $(\substack{n+d\\d})$.

A polynomial $f\in\mathbb{R}\left [  x\right ]$ is said to be a sum of squares (SOS)
if there exist polynomials $p_1,\cdots,p_k\in\mathbb{R}\left [  x\right ]$
such that $f=p_1^2+\cdots+p_k^2$. The set of all SOS polynomial in $\re[x]$
is denoted as $\Sigma[x]$.
For a degree $d$, we denote the truncation
\[
\Sigma\left [  x\right ]_{d}   \, \coloneqq \,
\Sigma\left [  x\right ] \cap  \re[x]_d .
\]
It is well known that (see \cite[Sec.~2.4]{mpo})
a polynomial $f \in \re [x]_{2d}$ is SOS if and only if
there exists $X \in \mc{S}_{+}^{N}$, with $N = \binom{n+d}{d}$, such that
\begin{equation}
f  =  \left [  x\right ]_{d}^{T}  \cdot X \cdot\left [  x\right ]_{d}.
\end{equation}
For a tuple $g=(g_1,\cdots,g_m)$, its quadratic module is the set
\[
\qmod{g} \, \coloneqq \,
\Sigma \left [ x \right ]+g_1\cdot\Sigma \left [ x \right ]+
\cdots+g_m\Sigma \left [ x \right ].
\]
When $g$ is empty, $\qmod{ \emptyset } =  \Sigma [x]$.
For a degree $d$, we denote the truncation (let $g_0=1$)
\[
\qmod{g}_{d} \, \coloneqq \,
\Big \{\sum_{i=0}^{m}\sigma _{i}g_{i} :\sigma_i\in \Sigma [x],\,
\deg(\sigma_ig_i ) \leq d  \Big \}.
\]

\subsection{Moment and localizing matrices}
\label{ssc:momloc}

The notation $\mathbb{R}^{\mathbb{N}^n_{d}}$ stands for the space of real vectors
$y$ that are indexed by $\alpha\in \mathbb{N}^n_{d}$, i.e.,
\[
y  = (y_{\alpha})_{\alpha\in\mathbb{N}^n_{d}}.
\]
Such a vector $y$ is called a truncated multi-sequence (tms) of degree $d$.
The tms $y$ is said to admit a Borel measure $\mu$ on $\re^n$ if
\[
y_{\alpha}=\int x^{\alpha }d \mu  \quad \text{for all}\, \,
 \alpha\in \mathbb{N}^n_{d}.
\]
If it exists, such $\mu$ is called a representing measure for $y$.
The support of $\mu$ is the smallest closed set $T$ such that
$\mu (\re^n \setminus T) = 0$, and we denote it by $\supp{u}$.
The measure $u$ is said to be supported in a set $K$ if $\supp{\mu} \subseteq K$.

For $f \in \re[x]_d$ and $y\in \mathbb{R}^{\mathbb{N}^n_{d}}$,
we define the bilinear operation
\[
\langle  f, y \rangle  \, \coloneqq \,
\sum_{\alpha\in\mathbb{N}^n_{d}}f_{\alpha}y_{\alpha} \quad \text{for} \quad
f=\sum_{\alpha\in\mathbb{N}^n_{d}}f_{\alpha}x^{\alpha}.
\]
For an integer $\ell \in [ 0, d/2]$,
the $\ell$th order moment matrix of $y$ is the symmetric matrix
$M_\ell [y]$ such that
\be \label{eq-mom-mat}
\langle p^2,y  \rangle  \, \coloneqq \,  \vec{p}^T \cdot
 M_{l}[y]  \cdot \vec{p}
\qquad \text{for all} \quad p \in \re[x]_\ell.
\ee
Here $\vec{p}$ denotes the coefficient vector of $p$,
listed in the graded lexicographical order.
Indeed, one can see that
\[
M_\ell [y] =(y_{\alpha+\beta})_{\alpha\in\mathbb{N}_\ell^n,\,\beta\in\mathbb{N}_\ell^n}.
\]
For instance, when $n=2$ and $\ell = 2$, we have
\[
M_2[y] = \left[
\begin{array}{cccccc}
y_{00} & y_{10} & y_{01} & y_{20} & y_{11} & y_{02} \\
y_{10} & y_{20} & y_{11} & y_{30} & y_{21} & y_{12} \\
y_{01} & y_{11} & y_{02} & y_{21} & y_{12} & y_{03} \\
y_{20} & y_{30} & y_{21} & y_{40} & y_{31} & y_{22} \\
y_{11} & y_{21} & y_{12} & y_{31} & y_{22} & y_{13} \\
y_{02} & y_{12} & y_{03} & y_{22} & y_{13} & y_{04} \\
\end{array} \right].
\]
%
%
For a polynomial $h \in \re[x]_d$, denote by $\mathscr{V}_{h}^{(d)}[y]$
the vector such that
\[
\left\langle h \cdot ( v^T[x]_s ), y \right\rangle \, =  \,
v^T   \mathscr{V}_{h}^{(d)}[y] ,
\]
for every column vector $v$
of length $\binom{n+s}{s}$, where $s = d - \deg(h)$.
The $\mathscr{V}_{h}^{(d)}[y]$ is called the localizing vector of $h$
and generated by $y$ \cite{mpo}.
For instance, when $h=1-x_1x_2$ and $n=2,d=4$, we have
\[
\mathscr{V}_{h}^{(4)}[y]  \, = \,
\bbm y_{00} - y_{11} \\
y_{10} - y_{21} \\   y_{01} - y_{12} \\
y_{20} - y_{31} \\   y_{11} - y_{22} \\    y_{02} - y_{13}  \ebm.
\]

Given a polynomial $f \in \re[x]$ and a finite set of polynomials $P \subseteq \re[x]$,
we consider the constrained optimization problem:
\be   \label{min:f:p=0}
\left\{
\baray{cl}
\min  & f(x)  \\
\st  &  p(x) = 0 \,\,\, \text{for}\,\,\, p \in P.
\earay
\right.
\ee
Its global optimizers can be computed by moment relaxations.
For $\ell = 1, 2, \ldots$, the $\ell$th order moment relaxation is
\be    \label{fp:mom}
\left\{ \baray{cl}
     \min   & \langle  f, y \rangle  \\
 \st  &  \mathcal{V}_{p }^{(2\ell)} [ y ]=0 \,\,\, \text{for}\,\,\, p \in P,  \\
\quad & M_\ell [ y ]\succeq 0,\\
\quad & y_0=1,\, y\in\mathbb{R}^{\mathbb{N}_{2\ell}^{\bar{n} } }.
\earay \right.
\ee
Its dual optimization problem is the $\ell$th order SOS relaxation
\be  \label{fp:sos}
\left\{
\baray{cl}
\max  & \gamma \\
\st  & f - \gamma  \in  \ideal{P}_{2\ell}
      +  \Sigma[x]_{2\ell}.
\earay \right.
\ee
The sequence of relaxations \reff{fp:mom}--\reff{fp:mom}
is often referenced as the Moment-SOS hierarchy for solving \reff{min:f:p=0}.
We refer to \cite{LasBk15,LasICM,Lau09,mpo} for more detailed introductions
to the Moment-SOS hierarchy.

\section{Rank-1 tensor completion and matrix recovery}
\label{sec:tc-matrix}

For a partially given tensor $\mA \in \re^{n_1 \times n_2 \times n_3}$,
let $\Omega$ be the set of indices $(i,j,k)$
such that the tensor entry $\mA_{ijk}$ is given.
It is a subset of $[ n_1 ] \times [ n_2]  \times [ n_3 ]$.
We write $\Omega$ as the union
\be \label{eq-omega}
\Omega \,  = \,  \bigcup_{k=1}^{n_3}  \Omega_k ,
\ee
where the set $\Omega_k$ is
\[
\Omega_k \, \coloneqq \, \left\{ (i_1,j_1,k),\,(i_2,j_2,k),\cdots(i_{m_k},j_{m_k},k)\right\} .
\]
Let $m_k=\left|\Omega_k \right|$, the cardinality of $\Omega_k$.
If $\mA = a \otimes b \otimes c$,
then for each $k = 1, \ldots, n_3$,
\begin{equation}\label{eq-rank}
\begin{bmatrix}
  \mathcal{A}_{i_1j_1k}\\
\mathcal{A}_{i_2j_2k}\\
\vdots \\
 \mathcal{A}_{i_{m_k}j_{m_k}k}
\end{bmatrix}
=\begin{bmatrix}
a_{i_1}b_{j_1}\\
a_{i_2}b_{j_2}\\
\vdots \\
a_{i_{m_k}}b_{j_{m_k}}
\end{bmatrix}c_k.
\end{equation}
The above implies $\rank \, B_k \le 1$, for the matrix
\[
B_k  \coloneqq  \begin{bmatrix}
 \mathcal{A}_{i_1j_1k} & a_{i_1}b_{j_1} \\
\mathcal{A}_{i_2j_2k} & a_{i_2}b_{j_2} \\
 \vdots & \vdots  \\
\mathcal{A}_{i_{m_k}j_{m_k}k} & a_{i_{m_k}}b_{j_{m_k}} \\
\end{bmatrix}.
\]
Clearly, $\rank \, B_k \le 1$ if and only if all
$2$-by-$2$ minors of $B_k$ are zeros, i.e.,
\begin{equation}\label{eq-1}
\det \bbm  \mA_{i_s j_s k} & a_{i_s}b_{j_s}  \\ \mA_{i_t j_t k} & a_{i_t}b_{j_t} \ebm  =
\mathcal{A}_{i_sj_sk}a_{i_t}b_{j_t}-\mathcal{A}_{i_tj_tk}a_{i_s}b_{j_s}=0,
\end{equation}
for all $1 \le s < t \le m_k$.
This gives $\binom{m_k}{2}$ quadratic equations.
There are totally $\sum_{k=1}^{n_3}\binom{m_k}{2}$ such equations.
In the above, we have $\rank \, B_k = 1$ if at least one entry of
$B_k$ is nonzero.

Let $(\hat{i}, \hat{j}, \hat{k})$
be the index of any tensor entry of largest absolute value:
\be \label{eq-label}
 |\mathcal{A}_{\hat{i}\hat{j}\hat{k}} |   \,\, =
 \max_{(i,\,j,\,k)\in \Omega} \left| \mathcal{A}_{{i}{j}{k}} \right|.
\ee
For a rank-$1$ completion $\mathcal{A} = a \otimes b \otimes c$ with nonzero $a,b$,
we can scale them such that $a_{\hat{i}}=b_{\hat{j}}=1$.
So, the rank-$1$ tensor completion is equivalent to
\be   \label{eq-rank-1}
\left\{
\baray{rl}
\mbox{find} & (a,\,b,\,c)\in\mathbb{R}^{n_1}\times \mathbb{R}^{n_2}\times \mathbb{R}^{n_3}\\
\st  &\mathcal{A}_{i_sj_sk}a_{i_t}b_{j_t}-\mathcal{A}_{i_tj_tk}a_{i_s}b_{j_s}=0  \\
       &  \qquad \text{for} \quad   (i_s, j_s, k), (i_t, j_t, k) \in \Omega_k, \, k = 1, \ldots, n_3,  \\
   &\mathcal{A}_{ijk}=a_i b_j c_k \quad   \text{for} \quad   (i,\,j,\,k)\in \Omega,  \\
  &a_{\hat{i}}=b_{\hat{j}}=1. \\
\earay
\right.
\ee

\subsection{Reduction to matrix recovery}
\label{ssc:red2mat}

If we let $X = ab^T$, then $\eqref{eq-1}$ is equivalent to
\begin{equation}  \label{eq-equa-matrix}
\mathcal{A}_{i_s j_s k}X_{i_t j_t}-\mathcal{A}_{i_t j_t k}X_{i_s j_s}=0
\end{equation}
for all feasible indices $s,t, k$.
If $\mathcal{A}$ has a rank-$1$ completion, i.e.,there is a pair $(a, b)$
satisfying \eqref{eq-1}, then the matrix $X = ab^T$ satisfies \reff{eq-equa-matrix}.
Therefore, we consider the matrix recovery problem
\be \label{eq-equivalence}
\left\{
\baray{cl}
\text{find}  &  X = ab^T \in\mathbb{R}^{n_1\times n_2}  \\
\st  & \mathcal{A}_{i_sj_sk}X_{i_tj_t}-\mathcal{A}_{i_tj_tk}X_{i_sj_s}=0   \\
   &  \qquad \text{for} \quad   (i_s,\,j_s,\,k), (i_t,\,j_t,\,k)\in \Omega_k,\,k=1,\ldots, n_3, \\
  &  X_{\hat{i}\hat{j}}=1,
\earay
\right.
\ee
where the pair $(\hat{i},\hat{j})$ is as in \reff{eq-label}.
The relationship between the tensor completion problem~\reff{eq-rank-1}
and the matrix recovery problem~\reff{eq-equivalence}
is given as follows.

\begin{theorem}\label{th-equivalence}
Suppose $\mA$ is the partially given tensor as above.
If $X=ab^T$ is feasible for \reff{eq-equivalence} and
for each  $k =1, \ldots, n_3$, the following equation
\begin{equation}  \label{eq-solve-ck}
\begin{bmatrix}
  \mathcal{A}_{i_1j_1k}\\
\mathcal{A}_{i_2j_2k}\\
\vdots \\
 \mathcal{A}_{i_{m_k}j_{m_k}k}
\end{bmatrix}
=\begin{bmatrix}
a_{i_1}b_{j_1}\\
a_{i_2}b_{j_2}\\
\vdots \\
a_{i_{m_k}}b_{j_{m_k}}
\end{bmatrix}c_k
\end{equation}
has a nonzero coefficient for $c_k$,
then $\mathcal{A}$ has the rank-$1$ completion $\mA = a \otimes b \otimes c$
with $c =(c_1, \ldots, c_{n_3})$. Conversely, if $\mathcal{A}$ is rank-$1$
completable and $\mA_{\hat{i} \hat{j} \hat{k} } \ne 0$,
then \reff{eq-equivalence} has a rank-$1$ feasible matrix solution $X=ab^T$.
\end{theorem}
\begin{proof}
If $X = ab^T$ is feasible for \reff{eq-equivalence}, then
$X_{\hat{i}\hat{j}} =  a_{\hat{i}} b_{\hat{j}} = 1$.
Up to a scaling, we can further assume
$a_{\hat{i}} = b_{\hat{j}} = 1$.
Moreover, $a$ and $b$ also satisfy the equation
\[
\mathcal{A}_{i_sj_sk}a_{i_t}b_{j_t}-\mathcal{A}_{i_tj_tk}a_{i_s}b_{j_s}=0
\]
for all $(i_s,\,j_s,\,k), (i_t,\,j_t,\,k)\in \Omega_k$
and $k=1,\ldots, n_3$, so $\rank\, B_k \le 1$.
This means that the equation \eqref{eq-solve-ck}
has a solution $c_k$ for each $k$,
since the coefficient vector is nonzero.
Let $c=(c_1,\cdots c_{n_3})$, then $a_ib_jc_k=\mathcal{A}_{ijk}$
for all $(i,\,j,\,k)\in \Omega$.
So, $a \otimes b \otimes c$ is a rank-$1$ tensor completion for $\mA$.

Conversely, if $\mA$ is rank-$1$ completable, say, $\mA =a \otimes b \otimes c$,
then for all $(i,\,j,\,k)\in \Omega$, $\mathcal{A}_{ijk}=a_ib_jc_k$.
Each $c_k$ is a solution for $\eqref{eq-solve-ck}$,
so $\rank\, B_k \le 1$ and \reff{eq-1} holds.
Let $X=a b^T$, then \reff{eq-equa-matrix} holds.
Since $\mA_{\hat{i} \hat{j} \hat{k} } \ne 0$ and $a, b$ are non-zero vectors,
we scale them such that $a_{\hat{i}} = b_{\hat{j}}=1$.
Then, $X_{\hat{i}\hat{j}}=1$ and hence
$X$ is a rank-$1$ feasible matrix solution for \reff{eq-equivalence}.
\end{proof}

\begin{remark}\label{re-no-solution}
\rm
(i) If the coefficient vector in the right-hand of $\eqref{eq-solve-ck}$
is zero and the left-hand side one is also zero,
then we can select arbitrary value for $c_k$.
In this case, the rank-$1$ tensor completion is not unique.
(ii) If the coefficient vector in the right-hand of $\eqref{eq-solve-ck}$
is zero but the left-hand side one is nonzero, then \reff{eq-solve-ck}
has no feasible solution and a rank-$1$ completion does not exist.
\end{remark}

\subsection{Strongly rank-1 completable tensors}
\label{ssc:strongrk1}

We consider a special class of rank-$1$ tensor completion problems
that can be reduced to rank-$1$ matrix completions.
Denote the index set
\begin{equation}  \label{eq-label-2}
\widetilde{\Omega}  \, \coloneqq \,
\left\{ (i,\,j): \,  (i,\,j,\,k)\in\Omega\right\} .
\end{equation}
It is the projection of $\Omega$ on the first two indices.
The equations of \reff{eq-equa-matrix} are homogenous
in the vector of partial matrix entries
\[
X_{ \widetilde{\Omega}  }   \, \coloneqq \,
(X_{ij})_{ (i,j)  \in \widetilde{\Omega}  } .
\]
Note that $X_{ \widetilde{\Omega}  }$ can be viewed as a vector in
$\re^{ \widetilde{\Omega} }$.
The set of solutions to the linear system \reff{eq-equa-matrix}
is a subspace of $\re^{ \widetilde{\Omega} }$.
We are interested in the case that \reff{eq-equa-matrix}
has a unique solution (up to scaling), i.e.,
the subspace of its solutions is one dimensional.
This leads to the following definition.

\begin{defi}\label{de-strongly}
The partially given tensor $\mathcal{A}$ is strongly rank-$1$ completable
if it has a rank-$1$ completion and the subspace of solutions
$X_{ \widetilde{\Omega}  }$ to \reff{eq-equa-matrix}
is one dimensional.
\end{defi}

Let $(\hat{i}, \hat{j},\,\hat{k})$ be the index as in \reff{eq-label}.
If $X_{ \widetilde{\Omega}  }$ is a nonzero solution to \reff{eq-equa-matrix},
we can scale it such that $X_{ \hat{i} \hat{j}  } = 1$.
If the subspace of solutions $X_{ \widetilde{\Omega}  }$ to \reff{eq-equa-matrix}
is one dimensional, there exist scalars $w_{ij}$ such that
\be \label{eq:wij}
X_{ \hat{i} \hat{j}  } = 1, \quad
X_{ij} = w_{ij} \quad \text{for} \,\, (i,j) \in  \widetilde{\Omega} .
\ee
So, when $\mathcal{A}$ is strongly rank-$1$ completable,
the problem \eqref{eq-equivalence} is equivalent to
the rank-$1$ matrix completion problem
\begin{equation}\label{eq-matrix-com}
\left\{
\baray{cl}
\text{find}  & X = a b^T \in\mathbb{R}^{n_1\times n_2}  \\
\st  & X_{ij} = w_{ij} \quad \text{for} \quad (i,\,j)\in \widetilde{\Omega}, \\
  &  X_{ \hat{i} \hat{j}  } = 1 .
\earay
\right.
\end{equation}
%
%

%
%
The rank-$1$ matrix completion problem~\reff{eq-matrix-com}
can be solved explicitly when the bipartite graph determined by
$\widetilde{\Omega}$ is connected (see \cite{C-D-17}).
Let $V_1$,$V_2$ be the sets:
\[
V_1 = \{i \in [n_1]:  (i,j) \in \widetilde{\Omega} \}, \quad
V_2 = \{j \in [n_2]:  (i,j) \in \widetilde{\Omega} \} .
\]
Consider the bipartite graph $G(V_1,\,V_2,\,\widetilde{\Omega})$ with vertex sets
$V_1,V_2$ and whose edge set is $\widetilde{\Omega}$.
When $G(V_1,\,V_2,\,\widetilde{\Omega})$ is connected, the rank-$1$ matrix
$X$ satisfying \reff{eq-matrix-com} is unique and
it can be determined by an iterative formula.
This can be seen as follows.
For $X_{ \hat{i} \hat{j} } = 1$, we can let
\[  a_{\hat{i}} = 1, \quad  b_{\hat{j}}=1 . \]
Since $G(V_1,\,V_2,\,\widetilde{\Omega})$ is connected,
there exist indices $i_1, j_1$ such that
$(i_1, \hat{j}) \in \widetilde{\Omega}$ and
$(\hat{i}, j_1) \in \widetilde{\Omega}$, so we can get
\be \nn
a_{i_1} = w_{i_1 \hat{j}}, \quad b_{j_1} = w_{\hat{i} j_1}.
\ee
Similarly, there exist indices $i_2, j_2$ such that
$(i_1, j_2) \in \widetilde{\Omega}$ and
$(i_2, j_1) \in \widetilde{\Omega}$, so
\be \nn
a_{i_2} = w_{i_2 j_1}/b_{j_1}, \quad b_{j_2} = w_{i_1 j_2}/a_{i_1}.
\ee
When the bipartite graph $G(V_1,\,V_2,\,\widetilde{\Omega})$ is connected,
repeating the above can produce edge connections:
\[
\begin{gathered}
 a_{i_1} \rightarrow b_{j_2} \rightarrow a_{i_3}\rightarrow b_{j_4}\rightarrow \cdots , \\
 b_{j_1} \rightarrow a_{i_2} \rightarrow b_{j_3}\rightarrow a_{i_4}\rightarrow \cdots  .
\end{gathered}
\]
So, the entries of $a,b$ can be given by the iterative formula
\be  \label{ail:bjl:ratio}
a_{i_{l+1}} = w_{i_{l+1} j_l}/b_{j_l}, \quad
b_{j_{l+1}} = w_{i_l j_{l+1}}/a_{i_l}, \quad l = 1, 2, \ldots .
\ee
When the graph $G(V_1,\,V_2,\,\widetilde{\Omega})$ is not connected,
it is a union of connected subgraphs.
We can do similar things for each of them.
For such a case, the rank one completion for $X$ is not unique.
We refer to \cite{C-D-17} for this.

The following is an example of applying the iterative formula
\reff{ail:bjl:ratio} to get rank-$1$ tensor completions.

\begin{example}
Consider the tensor $\mA \in \re^{ 3\times3\times 3 }$ with given entries:
\[
\begin{matrix}
\mathcal{A}_{111}=-1, & \mathcal{A}_{221}=-1, & \mathcal{A}_{311}=-1, & \mathcal{A}_{132}=1,  \\
\mathcal{A}_{312}=-1, & \mathcal{A}_{233}=1, & \mathcal{A}_{313}=1, & \mathcal{A}_{323}=-1.  \\
\end{matrix}
\]
We have $\widetilde{\Omega} =  \big \{
(1,\,1),\,(1,\,3),\,(2,\,2) ,\,(2,\,3),\,(3,\,1),\,(3,\,2)\big \}$.
The equation \reff{eq-equa-matrix} gives
\begin{equation}\label{eq-ex1-1}
\left\{
\begin{matrix}
\mathcal{A}_{111}X_{22}-\mathcal{A}_{221}X_{11}=0, &   \mathcal{A}_{111}X_{31}-\mathcal{A}_{311}X_{11}=0,\\
\mathcal{A}_{221}X_{31}-\mathcal{A}_{311}X_{22}=0, &   \mathcal{A}_{132}X_{31}-\mathcal{A}_{312}X_{13}=0,\\
\mathcal{A}_{233}X_{31}-\mathcal{A}_{313}X_{23}=0, &   \mathcal{A}_{233}X_{32}-\mathcal{A}_{323}X_{23}=0,\\
\mathcal{A}_{313}X_{32}-\mathcal{A}_{323}X_{31}=0. &   \quad\\
\end{matrix} \right.
\end{equation}
Then, $\eqref{eq-ex1-1}$ can be written as
\begin{equation}\label{eq-ex1-2}
\left[\begin{array}{rrrrrrr}
1 & 0 & -1 & 0  & 0  & 0 \\
1 & 0 & 0  & 0  & -1 & 0 \\
0 & 0 & 1  & 0  & -1 & 0 \\
0 & 1 & 0  & 0  & 1  & 0 \\
0 & 0 & 0  & -1 & 1  & 0 \\
0 & 0 & 0  & 1  & 0  & 1 \\
0 & 0 & 0  & 0  & 1  & 1  \\
\end{array}\right]
\begin{bmatrix}
X_{11} \\  X_{13}  \\  X_{22} \\  X_{23} \\  X_{31}  \\  X_{32}
\end{bmatrix}
=0 .
\end{equation}
The subspace of solutions for the above equation
is one dimensional, spanned by
\[
(1,\,-1,\,1,\,1,\,1,\,-1).
\]
Let $(\hat{i},\,\hat{j},\,\hat{k})=(1,\,1,\,1)$,
since $\left|\mathcal{A}_{ijk} \right|=1$ for all $(i,\,j,\,k)\in \Omega$.
So, we let $X_{11}=1$ and $a_1=b_1=1$.
The values of $X_{ij}$ and $w_{ij}$ in \reff{eq:wij} are:
\[
\begin{array}{lll}
   X_{11}=w_{11}=1, & X_{13}=w_{13}=-1, & X_{22}=w_{22}=1, \\
  X_{23}=w_{23}=1, & X_{31}=w_{31}=1, & X_{32}=w_{32}=-1.
\end{array}
\]
Then, we can get $a$, $b$ by the formula \eqref{ail:bjl:ratio} as:
\[
\begin{aligned}
 a_1 & =1 \xrightarrow[]{ w_{13}=-1  } b_3=-1 \xrightarrow[]{ w_{23}=1 } a_2=-1 \xrightarrow[]{\ w_{22}=1 } b_2=-1, \\
 b_1 & =1 \xrightarrow[]{ w_{31}=1 } a_3=1.
\end{aligned}
\]
Hence, we get
\[
a=(1,\,-1,\,1),\quad b=(1,\,-1,\,-1).
\]
Finally, we get $c=(-1,\,-1,\,1)$ by solving $\eqref{eq-solve-ck}$,
which gives the rank-$1$ tensor completion
$\mathcal{A} = a\otimes b\otimes c$.
\end{example}

\subsection{Nuclear norm relaxation}
\label{ssc:nuclear}

We discuss the case that the partially given tensor $\mA$ is not strongly rank-$1$ completable.
To find a feasible rank-$1$ matrix $X$ for \eqref{eq-equivalence},
a frequently used approach is to solve the nuclear norm relaxation
($\| X \|_*$ denotes the nuclear norm of $X$, i.e.,
the sum of all its singular values):
\be   \label{eq-matrec-3dim}
\left\{
\baray{cl}
\min   & \left\| X\right\|_{\ast}   \\
\st  &\mathcal{A}_{i_sj_sk}X_{i_tj_t}- \mathcal{A}_{i_tj_tk}X_{i_sj_s}=0  \\
             & \qquad \text{for} \quad  (i_s,\,j_s,\,k), (i_t,\,j_t,\,k)\in \Omega_k,\,k=1,\ldots, n_3, \\
\quad & X_{\hat{i}\hat{j}}=1 .
\earay
\right.
\ee
The matrix nuclear norm minimization
can be equivalently reformulated as a semidefinite program
(see \cite{candes2012,fazel2002}).
Indeed, up to applying the singular value decomposition of $X$,
 one can show that $\| X \|_{\ast}$ equals the minimum value of
\[
\left\{ \baray{cl}
 \min & \half \big( \mbox{Trace}(W_1)  + \mbox{Trace}(W_2)  \big) \\
\st & \bbm  W_1& X \\  X^T& W_2 \ebm \succeq 0.
\earay \right.
\]
Therefore, \reff{eq-matrec-3dim}
is equivalent to the semidefinite program
\begin{equation}\label{eq-sdp-matrec}
\left\{
\baray{cl}
\min         &  \mbox{Trace}(W_1)  + \mbox{Trace}(W_2)  \\
\st  &  \mathcal{A}_{i_sj_sk}X_{i_tj_t} - \mathcal{A}_{i_tj_tk}X_{i_sj_s}=0 \quad  \\
             & \qquad \text{for} \quad  (i_s,\,j_s,\,k), (i_t,\,j_t,\,k)\in \Omega_k,\,k=1,\ldots, n_3, \\
 \quad & \begin{bmatrix}
 W_1& X \\
 X^T& W_2 \\
\end{bmatrix}\succeq 0, \,
X_{\hat{i}\hat{j}}=1.
\earay
\right.
\end{equation}

We would like to remark that the nuclear norm relaxation \reff{eq-matrec-3dim}
sometimes returns a rank-$1$ matrix $X$ for \eqref{eq-equivalence},
while sometimes it does not. This is shown in the following example.

\begin{example}
(i) Consider the tensor $\mA \in \re^{ 4\times4\times 4}$ with given entries:
\[
\begin{matrix}
\mathcal{A}_{121}=2, & \mathcal{A}_{131}=4, & \mathcal{A}_{441}=1, & \mathcal{A}_{112}=4, \\
\mathcal{A}_{232}=4, & \mathcal{A}_{322}=2, & \mathcal{A}_{412}=2, & \mathcal{A}_{343}=1, \\
\mathcal{A}_{413}=1, & \mathcal{A}_{423}=1, & \mathcal{A}_{443}=1, & \mathcal{A}_{114}=2.  \\
\end{matrix}
 \]
Solving the nuclear norm relaxation \eqref{eq-matrec-3dim},
we get the optimal matrix
\[
X^*=\begin{bmatrix}
1 & 1 & 2 & 1 \\
\frac{1}{2} & \frac{1}{2} & 1 & \frac{1}{2} \\
\frac{1}{2} & \frac{1}{2} & 1 & \frac{1}{2}  \\
\frac{1}{2} & \frac{1}{2} & 1 & \frac{1}{2}  \\
\end{bmatrix} .
\]
It is rank-$1$ and $X^* = a^* {b^*}^T$, with
\[
a^* = ( 1, \,\frac{1}{2}, \,\frac{1}{2}, \,\frac{1}{2} ), \quad
b^* = (  1, \, 1,  \, 2,  \,  1 ) .
\]
Then, we get the vector $c^*=(2, \,4, \,2, \,2)$
by solving $\eqref{eq-solve-ck}$, which gives the rank-$1$ tensor completion
$\mA =a^* \otimes b^*  \otimes c^*$.

\noindent
(ii) Consider the tensor $\mA \in \re^{ 3\times  3  \times 3}$
with given entries:
\[
\mathcal{A}_{131}=4, \,\,
\mathcal{A}_{222}=2, \,\,
\mathcal{A}_{133}=4, \,\,
\mathcal{A}_{213}=1, \,\,
\mathcal{A}_{323}=1.
\]
Solving the nuclear norm relaxation \eqref{eq-matrec-3dim},
we get the optimal matrix
\[
X^*=\begin{bmatrix}
  0 & 0  & 1  \\
  \frac{1}{4} & 0  & 0  \\
  0 & \frac{1}{4}  & 0  \\
\end{bmatrix}.
\]
However, $\rank \,X^*=3>1$, so the nuclear norm relaxation \eqref{eq-matrec-3dim}
does not give a rank-$1$ tensor completion.
On the other hand, this tensor $\mA$ is rank-$1$ completable, since
\[
\mA \, = \, (1,\, \frac{1}{2},\, \frac{1}{2}) \otimes (1,\, 1,\, 2) \otimes (2,\, 4,\, 2).
\]
It is interesting to note that there exists a different tensor completion, e.g.,
\[
\mA \, = \, (1,\, \frac{1}{4},\, \frac{1}{4}) \otimes (1,\, 1,\, 1) \otimes (4,\, 8,\, 4).
\]
For this instance, the rank-$1$ tensor completion is not unique.
We expect at least $7$ given entries for it to be unique.
\end{example}

\section{The Moment-SOS Relaxations}
\label{sec:tc-mom}

The nuclear norm relaxation may not return a rank-$1$ tensor completion.
When this is the case, we propose a moment hierarchy of
semidefinite programming relaxations to solve the problem.
As shown in Section~\ref{sec:tc-matrix},  the rank-$1$ tensor completion
problem \eqref{eq-rank-1} is equivalent to the feasibility problem:
\begin{equation}\label{eq-poly}
\left\{
\baray{cl}
\text{find} & (a,b)\in \mathbb{R}^{n_1}\times \mathbb{R}^{n_2}  \\
\st & \mathcal{A}_{i_sj_sk}a_{i_t}b_{j_t}-\mathcal{A}_{i_tj_tk}a_{i_s}b_{j_s}=0  \\
& \qquad \text{for} \quad  (i_s,\,j_s,\,k), (i_t,\,j_t,\,k) \in \Omega_k,\,k = 1, \ldots,  n_3, \\
& a_{\hat{i}}=b_{\hat{j}}  = 1 .
\earay
\right.
\end{equation}
%
%
The above pair $(\hat{i},\, \hat{j})$ is as in \eqref{eq-label}.
To solve \eqref{eq-poly}, we select a coercive quadratic objective
\[
f(a, b) \, = \, \begin{bmatrix}
 a\\b
\end{bmatrix}^TF
\begin{bmatrix}
 a\\b
\end{bmatrix} .
\]
where $F$ is a symmetric positive definite matrix of order $(n_1+n_2)$.
Then, we consider the polynomial optimization problem:
\begin{equation}\label{cpo-s}
\left\{\baray{cl}
      \min    &  f(a,\,b)  \\
 \st  & \mathcal{A}_{i_sj_sk}a_{i_t}b_{j_t} -
      \mathcal{A}_{i_tj_tk}a_{i_s}b_{j_s}=0  \\
 & \qquad \text{for} \quad  (i_s,\,j_s,\,k), (i_t,\,j_t,\,k) \in \Omega_k,\,k = 1, \ldots,  n_3, \\
 &  a_{\hat{i}}  = b_{\hat{j}} =  1.
\earay\right.
\end{equation}
When $F$ is positive definite, the optimization problem \reff{cpo-s}
must have a minimizer if it is feasible.
Moreover, when $F$ is generically selected
(i.e., $F$ is chosen from a Zariski open set in the matrix space),
the minimizer of \reff{cpo-s} is unique.
We refer to \cite[Section~6.3]{HaPham} for uniqueness of
optimizers for polynomial optimization.

Assume $(a^*,\, b^*)$ is an optimizer for \eqref{cpo-s}.
If for each  $k  \in  [ n_3 ]$, the equation
 \begin{equation}\label{eq-solve-ck-5}
\begin{bmatrix}
  \mathcal{A}_{i_1j_1k}\\
\mathcal{A}_{i_2j_2k}\\
\vdots \\
 \mathcal{A}_{i_{m_k}j_{m_k}k}
\end{bmatrix}
=\begin{bmatrix}
a^*_{i_1}b^*_{j_1}\\
a^*_{i_2}b^*_{j_2}\\
\vdots \\
a^*_{i_{m_k}}b^*_{j_{m_k}}
\end{bmatrix}c_k^*
\end{equation}
has a nonzero coefficient for $c_k^*$,
then we get the rank-$1$ completion $\mathcal{A}=a^*\otimes b^*\otimes c^*$,
with $c^*=(c_1^*,\cdots,c_{n_3}^*)$.
This follows from Theorem~\ref{th-equivalence}.

For convenience of notation, denote the vector variable
\[
x \, \coloneqq \,  (a_1, \ldots, a_{\hat{i}-1}, a_{\hat{i}+1}, \ldots, a_{n_1},
b_1, \ldots, b_{\hat{j}-1}, b_{\hat{j}+1}, \ldots, b_{n_1} )
\in \re^{ \bar{n} } ,
\]
where $\bar{n} = n_1 + n_2 - 2$.
The objective $f(a,b)$ is a quadratic polynomial in $x$
and we write it as $f(x)$.
%
%
Denote the set of quadratic polynomials in $x$:
\be \label{polyset:Phi}
\Phi  = \left\{ \mathcal{A}_{i_sj_sk}a_{i_t}b_{j_t} -
\mathcal{A}_{i_tj_tk}a_{i_s}b_{j_s}
\left| \baray{c}
(i_s,\,j_s,\,k), (i_t,\,j_t,\,k) \in \Omega_k, \\
s < t, \, \, k = 1, \ldots,  n_3
\earay \right.
\right\} .
\ee
There are totally $\sum_{k=1}^{n_3} \binom{m_k}{2}$
polynomials in $\Phi$, where $m_k = | \Omega_k |$.
Then the optimization problem \reff{cpo-s}
can be rewritten as
\begin{equation} \label{pop:f(x)}
\left\{
\baray{cl}
\min  & f(x)  \\
\st  &  \phi(x) = 0 \quad
\text{for\, all}\,\,\, \phi \in \Phi.
\earay
\right.
\end{equation}

For $\ell =1,2,\ldots$, the $\ell$th
order moment relaxation for solving \eqref{pop:f(x)} is
\be  \label{dual-pro}
\left\{ \baray{cl}
     \min   & \langle  f, y \rangle  \\
 \st  &  \mathcal{V}_{\phi }^{(2\ell)} [ y ]=0
\quad \text{for\, all}\,\,\, \phi \in \Phi,  \\
\quad & M_\ell [ y ]\succeq 0,\\
\quad & y_0=1,\, y\in\mathbb{R}^{\mathbb{N}_{2\ell}^{\bar{n} } }.
\earay \right.
\ee
We refer to Subsection~\ref{ssc:momloc}
for the notation $\mathcal{V}_{p}^{2\ell}[y]$ and $M_\ell[ y ]$.
Note that the vector $y$ is indexed by monomial powers
$\af \in \mathbb{N}_{2\ell}^{ \bar{n} }$ and
$\mathcal{V}_{p}^{2\ell}[y], M_\ell[y]$ are linear in $y$.
The moment relaxation \eqref{dual-pro}
can be implemented in the software {\tt GloptiPoly~3} \cite{GloPol3}.

Suppose $y^*$ is an optimizer of $\eqref{dual-pro}$.
To extract a minimizer for \eqref{cpo-s},
one could consider the rank-$1$ condition:
there exists an integer $t \in  [1,\, \ell ]$ such that
\be   \label{eq-flat-truncation}
\rank \, M_t [y^*] \, = \, 1.
\ee
Note that the entries of $y^*$ are indexed by $\af \in \N^{\bar{n}}_{2\ell}$.
When \reff{eq-flat-truncation} holds, the vector
\[
x^* \, = \, (y^*_{e_1}, \ldots, y^*_{e_{\bar{n}}} )
\]
is a minimizer of \reff{pop:f(x)}.
We refer to \cite[Sec.~4.2]{mpo} for this fact.

\begin{theorem} \label{thm:nsym:cvg}
For the moment relaxation \eqref{dual-pro}, we have:
\begin{enumerate}

\item [(i)]
The partially given tensor $\mathcal{A}$ has no rank-$1$ completion
if and only if
the moment relaxation $\eqref{dual-pro}$ is infeasible for some order $\ell$.

\item [(ii)]
Suppose $V_{\re}(\Phi)$ is a nonempty finite set
or a nonempty compact smooth\footnote{
Here, $V_{\re}(\Phi)$ is said to be smooth if there exists a finite set
$P \subseteq \re[x]$ such that $\ideal{P} = \ideal{\Phi}$ and the gradient vector set
$\{ \nabla p(v) \}_{p \in P}$ is linearly independent for every $v \in V_{\re}(\Phi)$.}
variety. If the objective $f$ is generic\footnote{
This means that the coefficient vector of $f$ is chosen from
a Zariski open set in the embedding space.
}, then the moment relaxation \eqref{dual-pro} has minimizers
and each minimizer $y^*$ must satisfy the rank condition \eqref{eq-flat-truncation},
when $\ell$ is big enough.

\end{enumerate}
\end{theorem}
\begin{proof}
(i) If the moment relaxation \eqref{dual-pro} is infeasible for some order $\ell$,
then the feasible set of \eqref{cpo-s} must be empty.
This is because  \eqref{dual-pro} is a relaxation of \eqref{cpo-s}.
If the problem $\eqref{cpo-s}$ is infeasible,
then we have $-1  \in \ideal{\Phi} + \Sigma[x]$,
by Real Nullstellensatz
(see \cite[Theorem~2.6.3]{mpo}). So, it must hold that
\be \label{-1:sos}
-1  \, \in \, \ideal{\Phi}_{2\ell} + \Sigma[x]_{2\ell},
\ee
when $\ell$ is big enough.
The dual optimization of \eqref{dual-pro} is the SOS relaxation
\begin{equation}   \label{max}
\left\{
\baray{cl}
\max  & \gamma \\
\st  & f - \gamma  \in  \ideal{\Phi}_{2\ell}
      +  \Sigma[x]_{2\ell}.
\earay \right.
\end{equation}
The condition \reff{-1:sos} implies that the dual maximization problem
\reff{max} is unbounded above.
By weak duality, the moment relaxation \eqref{dual-pro}
must be infeasible.

(ii) First, consider the case that
the variety $V_{\re}(\Phi)$ is a nonempty finite set.
It is shown in \cite[Theorem~5.6.1]{mpo}
(also see \cite{Lau09} or \cite{317-po-r}) that
the moment relaxation \eqref{dual-pro} has minimizers
and each minimizer $y^*$ must satisfy the rank condition
\be \label{FT:nsy}
 \rank \, M_{t-1}  [ y^* ]  \,  =  \, \rank \, M_t [ y^* ]  ,
\ee
when $\ell$ big enough. Moreover, there are
$\rank \, M_t [y^*]$ minimizers for \reff{pop:f(x)},
which are contained in the support of the representing measure for the subvector
\[
y^*\vert_{2t} \, \coloneqq \, (y^*_\af)_{\af \in \N^{\bar{n}}_{2t} } .
\]
When coefficients of a polynomial are generically chosen
(i.e., the vector of its coefficients is chosen from a Zariski open set),
its optimizer is unique (see \cite[Section~6.3]{HaPham}).
Since the objective $f$ is generic,
the optimization problem \eqref{cpo-s} has a unique minimizer,
so the rank condition \reff{eq-flat-truncation} holds.

Second, consider the case that $V_{\re}(\Phi) \ne \emptyset$ is a smooth variety.
When $f$ is generic, \eqref{pop:f(x)} has a unique minimizer $u$.
Since $V_{\re}(\Phi)$ is smooth, there exists a finite set
$P \subseteq \re[x]$ such that $\ideal{P} = \ideal{\Phi}$ and the gradient vector set
$\{ \nabla p(v) \}_{p \in P}$ is linearly independent for every $v \in V_{\re}(\Phi)$.
So, the constraint in \eqref{pop:f(x)} is equivalent to that
$p(x) = 0$ for every $p \in P$. This means
the classical linear independence qualification condition holds at $u$
(see \cite[Section~5.1]{mpo}).
Since $f$ is generic, the second order sufficiency condition also holds at $u$
(see \cite[Section~5.5]{mpo}).
Since $V_{\re}(\Phi)$ is compact,
$\ideal{\Phi} + \Sigma[x]$ is archimedean.
Therefore, the moment hierarchy of semidefinite programming relaxations \eqref{dual-pro}
has finite convergence. It is shown in \cite[Theorem~5.4.2]{mpo} that
when $\ell$ is big enough,
the moment relaxation \eqref{dual-pro} has minimizers
and each minimizer $y^*$ must satisfy the rank condition \reff{FT:nsy}.
Moreover, there are $\rank \, M_t [y^*]$ minimizers for \eqref{pop:f(x)},
which are contained in the support of the representing measure for $y^*\vert_{2t}$.
Since $u$ is the unique minimizer,
the rank condition \eqref{eq-flat-truncation} holds.
\end{proof}

\section{Symmetric rank-1 tensor completions}
\label{sc:sym}

We discuss rank-$1$ tensor completions for symmetric tensors.
Note that a symmetric rank-$1$ tensor $\mA$ can be written as $a^{\otimes 3}$.
Therefore, we can set $a=b=c$ in the previous sections.
In \eqref{eq-rank-1}, the variable $b$ can be replaced by $a$.
Consequently, the nuclear norm relaxation and moment relaxation
have smaller sizes for the matrix variables.
Assume the dimension \[ n_1 = n_2 = n_3 = n . \]
For a partially given symmetric tensor $\mA  \in \mt{S}^3(\re^{n})$,
we still let $\Omega$ denote the set of $(i,j,k)$
such that the value of $\mA_{ijk}$ is given.
Since $\mA$ is symmetric, we can assume $\Omega$ is invariant under
permutations, i.e., for each $(i,j,k) \in \Omega$,
every permutation of $(i,j,k)$ also belongs to $\Omega$.
The index set $\Omega$ can still be decomposed as in \eqref{eq-omega}.

The cubic symmetric tensor $\mA$ is of rank-$1$ if and only if
$\mA =  a^{\otimes 3}$, for some vector $0 \ne a \in \re^{n}$.
Let $(\hat{i}, \hat{j}, \hat{k})$ be the index as in \eqref{eq-label}.
Then, we can write $a$ as
\be   \label{label-v}
 a =\sqrt[3]{\tau }v,  \quad   v_{\hat{i}} v_{\hat{j}} = 1, \quad
 \tau \in \re, \,\,  v \in \re^{n}.
\ee
Therefore, we can look for $(v,\tau)$ satisfying the system:
\be   \label{eq-rank-1-sy}
\left\{ \baray{l}
\mathcal{A}_{i_sj_sk}v_{i_t}v_{j_t}-\mathcal{A}_{i_tj_tk}v_{i_s}v_{j_s}=0  \\
  \qquad \text{for} \quad  (i_s,\,j_s,k),  (i_t,\,j_t,k) \in \Omega_k,\, k = 1, \ldots, n, \\
\mathcal{A}_{ijk}=\tau v_iv_jv_k  \quad \text{for} \,\,\,  (i,\,j,\,k)\in \Omega, \\
v_{\hat{i}} v_{\hat{j}}  =  1,\,\, \tau\in\mathbb{R} .
\earay
\right.
\ee
Let $V=vv^T$ be the symmetric matrix variable,
then \eqref{eq-rank-1-sy} is equivalent to
\begin{equation}\label{eq-equivalence-sy}
\left\{\baray{cl}
\text{ find } & V = vv^T \in\mathbb{R}^{n \times n} \\
\st      & \mathcal{A}_{i_sj_sk}V_{i_tj_t}-\mathcal{A}_{i_tj_tk}V_{i_sj_s}=0  \\
      &  \qquad \text{for} \quad  (i_s,\,j_s,k),  (i_t,\,j_t,k) \in \Omega_k,\, k = 1, \ldots, n, \\
      & V_{\hat{i}\hat{j}}=1,\,\, V\in \mathcal{S}_{+}^{n}.
\earay \right.
\end{equation}
In the above, $\mathcal{S}_{+}^{n}$ denotes the cone of all
$n$-by-$n$ real symmetric positive semidefinite matrices.

When $\mathcal{A}$ is strongly rank-$1$ completable,
\reff{eq-equivalence-sy} can be reduced to the rank-$1$ matrix completion problem:
\begin{equation}  \label{eq-mc-sy}
\left\{\baray{cl}
\text{find} & V \in\mathbb{R}^{n  \times  n} \\
\st  & \rank\, V = 1, \\
 & V_{ij} = w_{ij} \quad \text{for} \,\,\, (i,\,j)\in \widetilde{\Omega}, \\
  & V \in \mathcal{S}_{+}^{n},
\earay  \right.
\end{equation}
where the values $w_{ij}$ are given in \reff{eq:wij}.
When the bipartite graph $G(V_1,\,V_2,\,\widetilde{\Omega})$ is connected,
the rank-$1$ matrix $V=vv^T$ can be similarly found
by the iterative formula \reff{ail:bjl:ratio}.

When $\mathcal{A}$ is not strongly rank-$1$ completable,
we can look for a rank-$1$ matrix solution for \eqref{eq-equivalence-sy}
by solving the nuclear norm relaxation.
Since $V$ is symmetric positive semidefinite,
the nuclear norm $\| V \|_*  = \mbox{Trace}(V)$.
So, the nuclear norm relaxation of \reff{eq-equivalence-sy} is
\begin{equation}\label{eq-sy}
\left\{\baray{cl}
\min  &  \mbox{Trace}(V)  \\
\st  & \mA_{i_s j_s k}V_{i_t j_t}-\mathcal{A}_{i_t j_t k}V_{i_s j_s}=0  \\
        &  \qquad \text{for} \quad (i_s,\,j_s,k),  (i_t,\,j_t,k) \in \Omega_k,\, k = 1, \ldots, n, \\
       & V_{\hat{i}\hat{j}}=1, \,\,  V\in \mathcal{S}_{+}^{n} .
\earay \right.
\end{equation}
The nuclear norm relaxation~\reff{eq-sy} may or may not produce
a rank-$1$ matrix completion (see Example~\ref{ex-com-nuc-sy}).

When \reff{eq-sy}  fails to return a rank-$1$ matrix,
we can find a symmetric rank-$1$ tensor completion by solving moment relaxations.
Note that \reff{eq-rank-1-sy} can be equivalently formulated as
\be   \label{eq-poly-sy}
\left\{\baray{l}
\mathcal{A}_{i_s j_s k}v_{i_t}v_{j_t} -
    \mathcal{A}_{i_t j_t k}v_{i_s}v_{j_s}=0 \\
\qquad \text{for} \quad (i_s,j_s, k), (i_t,j_t, k) \in \Omega_k, \,  k = 1, \ldots,  n, \\
    v_{\hat{i}}=1  .
\earay  \right.
\ee
To get a feasible point $v$ for \reff{eq-poly-sy},
we select a quadratic objective
$f(v) = v^T F v.$
A feasible solution for \eqref{eq-poly-sy}
can be found by solving
\begin{equation}\label{eq-cpo-sy}
\left\{\baray{cl}
 \min  &  f(v)  \\
 \st &  \mathcal{A}_{i_s j_s k}v_{i_t}v_{j_t}-\mathcal{A}_{i_t j_t k}v_{i_s}v_{j_s}=0 \\
    &  \qquad \text{for} \quad  (i_s,\,j_s,\,k), (i_t,\,j_t,\,k) \in \Omega_k, \,  k = 1, \ldots,  n. \\
    &   v_{ \hat{i} } =  1.
\earay \right.
\end{equation}
When $F$ is a generic positive definite matrix,
the optimization problem \reff{eq-cpo-sy} must have a minimizer
and the minimizer is unique, if \reff{eq-poly-sy} is feasible.
For cleanness of the paper, we omit the details and refer to Section~\ref{sec:tc-mom}.

\section{Extension to higher order tensors}
\label{sc:highord}

The previous sections focus on rank-$1$ completions for cubic tensors.
However, the proposed methods can be naturally extended to higher order tensors.
A tensor $\mA \in  \re^{n_1\times\cdots \times n_d}$
can be reshaped as a tensor of order $3$.
For instance, when $d=4$, one can reshape $\mA$ as the cubic tensor
$\mB  \in \re^{n_1\times n_2\times (n_3 n_4) }$ such that
\[
\mA_{ijkl}  \, =  \,  \mB_{ij \ell}
\quad \text{for} \quad \ell = (k-1)n_4 + l.
\]
The resulting $\mB$ is a partially given cubic tensor.
We can apply our earlier proposed methods to get a rank-$1$ completion
\[
\mB \,=\, a\otimes b\otimes \hat{c}
\in \mathbb{R}^{n_1\times n_2\times (n_3 n_4)} .
\]
If $\hat{c}$ can also be reshaped as a rank-$1$ matrix, say, $cd^T \in \re^{n_3 \times n_4}$,
then the above produces a rank-$1$ tensor completion
$\mA = a \otimes b \otimes c \otimes d$.
However, if $\hat{c}$ cannot be reshaped as a rank-$1$ matrix,
then this does not give a rank-$1$ completion for $\mA$.

\begin{example}
(i) \label{ex-order-4.2}
Consider the tensor $\mA \in\mathbb{R}^{2\times 2\times 2\times 2}$
with given entries:
\[
\begin{array}{lll}
 \mA_{2211}=12, & \mA_{1121}=4, & \mA_{2121}=8, \\
 \mA_{2221}=4, & \mA_{1112}=6, & \mA_{2222}=2 . \\
 \end{array}
\]
It can be reshaped as the cubic tensor $\mB\in\mathbb{R}^{2\times 2\times 4}$ such that
\[
\begin{array}{lll}
 \mB_{221}=12, & \mB_{112}=4, & \mB_{212}=8, \\
 \mB_{222}=4, & \mB_{113}=6, & \mB_{224}=2 . \\
 \end{array}
\]
By solving the moment relaxation \reff{dual-pro}, we get
$\mB  =  a^*\otimes b^*\otimes \hat{c}$, with
\[
\begin{array}{lcl}
   a^*  = (\frac{1}{2} ,\,1 ), \quad b^*\,= \,(2  ,\,1 ), \quad
\hat{c} = (12 ,\,4 ,\,6  ,\,2  ).
\end{array}
\]
The vector $\hat{c}$ can be reshaped to the rank-$1$ matrix $c^* {d^*}^T$ with
\[
c^*=(12,\,4),\quad d^*=(1,\, \frac{1}{2}) .
\]
Therefore, we get the rank-$1$ completion
$\mA = a^*\otimes b^*\otimes c^*\otimes d^*$.\\
(ii) \label{ex-order-4.1}
Consider the tensor $\mA \in \re^{3\times 3\times 3\times 3}$
with given entries:
\[
\begin{array}{lllll}
 \mA_{1311}=6, & \mA_{3111}=24, & \mA_{3211}=12, & \mA_{3311}=12, & \mA_{2321}=4,  \\
 \mA_{2231}=8, & \mA_{3231}=8, & \mA_{3331}=8, & \mA_{2122}=4 , & \mA_{2322}=2,  \\
 \mA_{3122}=4, & \mA_{2313}=18, & \mA_{1223}=3, & \mA_{2233}=12 , & \mA_{2333}=12.  \\
 \end{array}
\]
It can be reshaped as the cubic tensor $\mB \in \mathbb{R}^{3\times 3\times 9}$ such that
\[
\begin{array}{lllll}
 \mB_{131}=6, & \mB_{311}=24, & \mB_{321}=12, & \mB_{331}=12, & \mB_{232}=4,\\
 \mB_{223}=8, & \mB_{323}=8, & \mB_{333}=8, & \mB_{215}=4, & \mB_{235}=2,  \\
 \mB_{315}=4, & \mB_{237}=18, & \mB_{128}=3, & \mB_{229}=12, & \mB_{239}=12.  \\
 \end{array}
\]
By solving the moment relaxation \reff{dual-pro}, we get
$\mB  =  a^*\otimes b^*\otimes \hat{c}$, with
\[
 a^* = (1,\,2,\,2),\quad b^* \,=\, (2,\,1,\,1), \quad
 \hat{c} = (6,\,2,\,4,\,0,\,1,\,0,\,9,\,3,\,6).
\]
There are no equations for $c_4$ and $c_6$, so we give zero values for them.
The vector $\hat{c}$ cannot be reshaped to a rank-$1$ matrix,
so the above does not produce a rank-$1$ tensor completion.
However, we remark that this tensor $\mA$ is rank-$1$ completable,
e.g., $\mA = a\otimes b\otimes c\otimes d$ with
\[
 a=(1,\,2,\,2),\quad b=(2,\,1,\,1), \quad
 c=(3,\,1,\,2),\quad d=(2,\,1,\,3).
\]
\end{example}

\section{Numerical experiments}
\label{sc:num}

In this section, we present numerical experiments for
getting rank-$1$ tensor completions by our proposed methods.
The computations are implemented in MATLAB
R2022b, on a desktop PC with CPU @2.10GHz and RAM 16G.
The numerical examples are solved by the
software Gloptipoly~3 \cite{GloPol3} and SeDuMi \cite{sturm1999}.
All computational results are displayed in four decimal digits,
for cleanness of the paper.
%
%

First, we remark that a rank-$1$ tensor completion
$\mA = a \otimes b \otimes c$ can be obtained by solving
the unconstrained polynomial optimization:
\be  \label{tc}
\min_{a,b,c} \,\, \sum_{(i,j,k)\in\Omega}(\mA_{ijk}-a_ib_jc_k)^2.
\ee
However, solving \reff{tc} is much more expensive than solving \eqref{cpo-s}
by using Moment-SOS relaxations.
This is because \reff{tc} has three vector variables $a,b,c$
and the polynomial has degree six.

\begin{example} \label{nls:sos}
We compare the numerical performance of solving \reff{tc} and \eqref{cpo-s}
by Moment-SOS relaxations, to get rank-$1$ tensor completions.
We randomly generate partially given rank-$1$ tensors $\mA$,
for which 20\% of its entries are given.
\begin{table}[ht!]
\centering
\caption{Computational performance of \reff{tc} and \eqref{cpo-s}. }
\label{ta-un}
\begin{tabular}{|c|c|c|c|}
\hline
 $n_1 = n_2 = n_3$     &   2   &     3   &    4      \\ \hline
time (s) for \reff{tc}     &  0.42 & 23.89   &  2945.13  \\ \hline
time (s) for \reff{cpo-s}  &  0.03 & 0.27    &   0.48    \\ \hline
\end{tabular}
\end{table}
The computational time (in seconds) for solving them by moment relaxations
are shown in Table~\ref{ta-un}. It is clear to see that
solving \eqref{cpo-s} is much more efficient than solving \reff{nls:sos}.
\end{example}

\subsection{Performance of nuclear norm relaxations}
\label{ssc:num:nuclear}

In this subsection, we explore the performance of nuclear norm relaxations
for getting rank-$1$ tensor completions.

\begin{example}
Let $n_1=n_2=n_3 = n$. We randomly generate rank-$1$ tensors
$\mathcal{A}\in \mathbb{R}^{n\times n\times n}$ and
randomly select the index set $\Omega$ of known entries.
The density of known tensor entries is measured as
\[ den \, = \, |\Omega| / n^3 . \]
According to \cite[Theorem~1.1]{Breiding23},
when the number of given tensor entries is at least $3n-1$ (i.e., $| \Omega | \ge 3n-1$),
the rank-$1$ tensor completion problem is expected to be identifiable
(i.e., the completion is unique when $\mA$ is a generic rank-$1$ tensor).
So we also record the oversampling rate
\[
 \rho \, = \, |\Omega| / (3n-1) .
\]
For each dimension $n$, we generate $20$ random instances.
The success rate is measured as the percentage of successful instances,
for which a rank-$1$ tensor completion is obtained by
solving the nuclear norm relaxation \reff{eq-matrec-3dim}.
We report the minimum density, for which the success rate
is higher than or equal to $90\%$.
The computational results are shown in Table~\ref{ta-nuc-nonsy}.
The computational time (in seconds) is reported as the average time
for solving \reff{eq-matrec-3dim} by using the software {\tt  SeDuMi}.
\begin{table}[htb]
\centering
\caption{The minimum density for the success rate of the
nuclear norm relaxation \reff{eq-matrec-3dim} to be at least $90\%$.}
\label{ta-nuc-nonsy}
\begin{tabular}{| c | c | c| c || c | c | c |c |   }
   \hline
  $n$ & $den$ & $\rho$ & time (s) &  $n$  & $den$ & $\rho$ & time (s)    \\		
  \hline
    3 & 62\%  & 2.10 & 0.01    &  12   & 15\%  &7.41  & 4.48  \\ \hline
    4 & 41\% & 2.39 & 0.01    &  13   & 14\%   & 8.09 & 8.79  \\ \hline
    5 & 37\% & 3.30 & 0.03    &  14   & 11\%   & 7.36 & 9.84  \\ \hline
    6 & 30\% & 3.81 & 0.04    &  15   & 10\%  & 7.67  & 15.85  \\ \hline
    7 & 25\% & 4.29 & 0.07    &  16   & 9\%  & 7.84  & 27.16  \\ \hline
    8 & 24\% & 5.34 & 0.25    &  17   & 9\%  & 8.84  & 48.52  \\ \hline
    9 & 19\% & 5.33 & 0.43    &  18   & 8\%  & 8.80  & 45.54  \\ \hline
   10 & 18\% & 6.21 & 1.02    &  19   & 8\% & 9.80   & 75.76  \\  \hline
   11 & 16\% & 6.66 & 1.79    &  20   & 8\%  & 10.85  & 164.62 \\ \hline
\end{tabular}
\end{table}
\end{example}

\begin{example}\label{den-sy-per}
Let $n=n_1=n_2=n_3$. We randomly generate rank-$1$ symmetric tensors
$\mathcal{A}\in \mathbb{R}^{n\times n\times n}$.
The index set $\Omega$ of known tensor entries is
generated randomly such that it is invariant under permutations, since $\mA$ is symmetric.
The density of known tensor entries is still measured as
$den = |\Omega| / n^3$. For each dimension $n$,
we also generate $20$ random instances.
The success rate is still measured as the percentage of successful instances,
for which a rank-$1$ tensor completion is obtained by
solving the symmetric nuclear norm minimization \reff{eq-sy}.
We report the minimum density for which the success rate
is at least $90\%$.
The computational results are shown in Table~\ref{ta-nuc-sy}.
Comparing Table~\ref{ta-nuc-nonsy} and Table~\ref{ta-nuc-sy},
one can see that solving \reff{eq-sy} is more efficient than solving
\reff{eq-matrec-3dim} for tensors of the same dimension.

\begin{table}[htb]
\centering
\caption{The minimum density for the success rate of
the symmetric nuclear norm minimization \reff{eq-sy}
to be at least $90\%$.}
\label{ta-nuc-sy}
\begin{tabular}{| c | c | c | |c | c | c | }
    \hline
   $n$ & $den$ & time (s) & $n$  & $den$ & time (s)    \\	
 \hline
   5   & 42\%  & 0.01      & 30    & 5\%   & 0.26 \\	
   \hline
  10   & 19\%  & 0.02      & 35    & 3\%   & 1.69 \\
  \hline
  15   & 13\%  & 0.07      & 40    & 2\%   & 2.01 \\
  \hline
  20   & 7\% & 0.13       & 45    & 2\%   & 6.53\\
  \hline
  25   & 5\% & 0.25        & 50   & 2\%   & 8.06 \\
 \hline
\end{tabular}
\end{table}
\end{example}

\subsection{Performance for strongly rank-1 completable tensors}

In this subsection, we give numerical experiments
for strongly rank-$1$ completable tensors. For such cases,
the problem can be reduced to matrix completion.
When the bipartite graph $G(V_1,\,V_2,\,\widetilde{\Omega})$ is connected,
the matrix completion problem~\reff{eq-matrix-com}
can be solved by the iterative formula \reff{ail:bjl:ratio}.
Comparing Table~\ref{ta-nuc-nonsy} and Table~\ref{ta-matrix-nonsy-2},
we can see that doing this is much faster
than solving the nuclear norm relaxation~\eqref{eq-matrec-3dim}
or moment relaxation~\reff{dual-pro}.

\begin{example}
Let $n_1=n_2=n_3 = n$.
We randomly generate partially given
rank-$1$ tensors $\mA \in \re^{n \times n \times n}$
and randomly select the index set $\Omega$ such that $\mathcal{A}$
is strongly rank-1 completable, in the following way.
First, select $i_1 \in [n]$, $j_1 \in [n]$ randomly
and initialize $\widetilde{\Omega} \coloneqq \{ (i_1, j_1) \}$.
Second, we randomly select $i_2 \in [n] \setminus \{i_1\}$
and $j_2 \in [n] \setminus \{j_1\}$, then update
$
\widetilde{\Omega} \coloneqq   \widetilde{\Omega}
\cup \{ (i_2,\,j_1), \, (i_1, j_2) \} .
$
We repeat doing this, until
$
\{i_1, \ldots, i_n \}  =  \{j_1, \ldots, j_n \}  =   [n] .
$
This gives two paths of connecting edges:
\[
\begin{gathered}
 a_{i_1} \rightarrow b_{j_2} \rightarrow a_{i_3}\rightarrow b_{j_4}\rightarrow
  \cdots \rightarrow a_{i_{n-1}} \rightarrow b_{j_{n}} ,  \\
 b_{j_1} \rightarrow a_{i_2} \rightarrow b_{j_3}\rightarrow a_{i_4}\rightarrow
  \cdots \rightarrow b_{j_{n-1}} \rightarrow a_{i_{n}} .
\end{gathered}
\]
The bipartite graph $G(V_1,\,V_2,\,\widetilde{\Omega})$
given by $\widetilde{\Omega}$ is connected.
After this is done, we generate $\Omega$ as follows.
For each $(i,j) \in \widetilde{\Omega}$, we randomly select $k \in [n]$
and let $(i,j, k) \in \Omega$.
Then, check the dimension of the solution subspace of $\eqref{eq-equa-matrix}$.
If it is bigger than one, we randomly select
a new triple $(i,j, k) \not\in \Omega$ with $(i,j)\in \widetilde{\Omega}$,
and let $\Omega \coloneqq \Omega \cup \{(i,j, k) \}$.
Repeat doing this, until the solution subspace of
$\eqref{eq-equa-matrix}$ has dimension one.
For the index set $\Omega$ generated this way,
the partially given tensor $\mathcal{A}$ is strongly rank-1 completable.
For all generated random instances, we successfully obtained a rank-$1$ tensor completion
by applying the iterative formula \reff{ail:bjl:ratio}.
The computational time (in seconds) is shown in Table~\ref{ta-matrix-nonsy-2}.
\begin{table}[htb]
\centering
\caption{Computational time for strongly rank-$1$ completable tensors
with the iterative formula \reff{ail:bjl:ratio}. }
\label{ta-matrix-nonsy-2}
\begin{tabular}{| c | c | c | c | c | c |   }    \hline
$n$         & 10    & 20    & 30    & 40    & 50    \\  \hline
time\,(s)   & 0.02  & 0.03  & 0.06  & 0.06  & 0.08  \\  \hline
$n$         & 60    & 70    & 80    & 90    & 100   \\	 \hline
time\,(s)   & 0.12  & 0.15  & 0.19  & 0.25  & 0.31  \\ \hline
\end{tabular}
\end{table}
\end{example}

\subsection{Performance of moment relaxations}
\label{ssc:num:mom}

For some partially given tensors, the nuclear norm relaxation may fail to give
a rank-$1$ tensor completion. However, by solving moment relaxations,
we can always get one, if it exists. This is shown in the following examples.

\begin{example}
Consider the tensor $\mA \in \re^{ 5\times 5\times 5}$ with given entries:
\[
\begin{matrix}
\mA_{111}=3, & \mA_{151}=1, & \mA_{241}=4, & \mA_{421}=1, & \mA_{451}=1,\\
\mA_{521}=1, & \mA_{522}=1, & \mA_{542}=2, & \mA_{234}=4, & \mA_{414}=6.\\
\end{matrix}
\]
For the nuclear norm relaxation~\reff{eq-matrec-3dim} for symmetric tensors,
the optimal matrix is
\[
X^*=\begin{bmatrix}
1.0000 & 0.2648  & 0.0625 & 0.5296  & 0.3333 \\
0.4724 & 0.6667  & 0.1575 & 1.3333  & 0.6667  \\
0.0000 & 0.0000  & 0.0000 & 0.0000  & 0.0000  \\
0.2362 & 0.3333  & 0.0787 & 0.6667  & 0.3333 \\
0.2362 & 0.3333  & 0.0787 & 0.6667  & 0.3333 \\
\end{bmatrix}.
\]
Since $\rank \,X^*=2 > 1$, it does not produce a rank-$1$ tensor completion.
However, by solving the moment relaxation \reff{dual-pro},
we can get the rank-$1$ tensor completion
$\mA  =  a^*\otimes b^*\otimes c^*$, with
\[
 a^* = (1,\, 2,\, 3,\, 1,\, 1), \quad
 b^* = (1,\,\frac{1}{3},\, \frac{1}{3},\, \frac{2}{3},\, \frac{1}{3}), \quad
 c^* = (3,\, 3,\,0,\, 6,\, 0) .
\]
There are no equations for $c_3$ and $c_5$, so we give zero values for them.
The computation took around $4.82$ seconds.
\end{example}

\begin{example}\label{ex-com-nuc-sy}
Consider the symmetric tensor $\mA \in \mt{S}^3( \re^5 )$ with given entries:
\[
\baray{llll}
 \mA_{151}=2,  & \mA_{221}=9,  & \mA_{541}=10, & \mA_{222}=27,\\
 \mA_{333}=64, & \mA_{353}=32, & \mA_{513}=8,  & \mA_{543}=40 ,\\
 \mA_{454}=50, & \mA_{115}=2,  & \mA_{235}=24. & \quad
\earay
\]
For the nuclear norm relaxation~\reff{eq-sy}, the optimal matrix is
\[
V^*=\begin{bmatrix}
      0.1961 & 0.2547 & 0.7845 & 0.9806 & 1.0000 \\
      0.2547 & 4.5000 & 2.3534 & 1.2739 & 1.2991 \\
      0.7845 & 2.3534 & 8.0000 & 3.9222 & 4.0000 \\
      0.9806 & 1.2739 & 3.9222 & 4.9028 & 5.0000 \\
      1.0000 & 1.2991 & 4.0000 & 5.0000 & 5.0991 \\
    \end{bmatrix}.
\]
Since $\rank \, V^* =3 > 1$, it does not produce a rank-$1$ completion.
However, by solving the moment relaxation,
we can get the symmetric rank-$1$ tensor completion
$\mA  = (v^*)^{\otimes 3}$, with
$v^*=(1,\, 3,\,4,\, 5,\, 2).$
The computation took around $0.14$ seconds.
\end{example}

\section{Conclusions}
\label{sc:con}

This paper studies the rank-$1$ tensor completion problem for cubic tensors.
It aims at finding missing values of a partially given tensor
so that it is of rank-$1$.
We reformulate this problem equivalently as
a special rank-$1$ matrix recovery problem,
which looks for a rank-$1$ matrix
satisfying a set of linear equations.
We propose both nuclear norm relaxation and moment relaxation methods
for solving the resulting rank-$1$ matrix recovery problem.
The nuclear norm relaxation sometimes get a rank-$1$ tensor completion,
but sometimes it does not.
The moment relaxation always get a rank-$1$ tensor completion
or detect its nonexistence.
For computational comparison, the nuclear norm relaxation approach
solves relatively larger problems,
while the moment relaxation approach solves relatively smaller ones.
For the special class of strongly rank-$1$ completable tensors,
the problem can be reduced to a rank-$1$ matrix completion problem.
When the corresponding bipartite graph is connected,
a rank-$1$ tensor completion can be obtained
by applying an iterative formula.
For strongly rank-$1$ completable tensors,
much larger problems can be solved efficiently.
Numerical experiments are provided to demonstrate
the efficiency of these proposed methods.

\bibliographystyle{siamplain}

\end{document}